\documentclass[12pt]{amsart}
\usepackage[english]{babel}
\usepackage{amsmath,amsthm,amssymb,amsfonts,amsthm}
\usepackage[pdftex]{color}
\usepackage[dvipsnames]{xcolor}
\usepackage[bookmarks=true,hyperindex,pdftex,colorlinks,citecolor=blue,linkcolor=blue,urlcolor=blue]{hyperref}
\usepackage{graphicx}
\usepackage{mathtools}
\usepackage{mathabx}
\usepackage{cancel}
\usepackage{framed}
\usepackage{float}

\newtheorem{theorem}{Theorem}[section]
\newtheorem{lemma}[theorem]{Lemma}

\newtheorem{proposition}[theorem]{Proposition}

\newtheorem{corollary}[theorem]{Corollary}

\theoremstyle{definition}
\newtheorem{definition}[theorem]{Definition}
\newtheorem{example}[theorem]{Example}
\newtheorem{remark}[theorem]{Remark}
\theoremstyle{remark}

\numberwithin{equation}{section}

\newcommand{\na}{\operatorname{NA}}

\newcommand{\amp}{\operatorname{AMp}}
\newcommand{\call}{\mathcal{L}}
\newcommand{\calf}{\mathcal{F}}
\newcommand{\calk}{\mathcal{K}}
\newcommand{\bbk}{\mathbb{K}}
\newcommand{\bbc}{\mathbb{C}}

\newcommand{\bbr}{\mathbb{R}}
\newcommand{\ma}{\operatorname{MA}}
\newcommand{\qna}{\operatorname{QNA}}
\newcommand{\qma}{\operatorname{QMA}}

\newcommand{\id}{\operatorname{Id}}
\newcommand{\bl}{\mathcal{BL}}
\newcommand{\iso}{\mathcal{ISO}}
\newcommand{\bbn}{\mathbb{N}}
\newcommand{\Lo}{\mathbf{L}_o}
\newcommand{\Loo}{\mathbf{L}_{o,o}}

\newcommand{\sss}{\mathcal{SS}}
\newcommand{\conv}{\operatorname{conv}}
\newcommand{\aconv}{\operatorname{aconv}}

\newcommand{\vertiii}[1]{{\left\vert\kern-0.25ex\left\vert\kern-0.25ex\left\vert #1 
\right\vert\kern-0.25ex\right\vert\kern-0.25ex\right\vert}}

\begin{document}
\title[On density and BPBp type properties for the minimum norm]{On density and Bishop-Phelps-Bollob\'as type properties for the minimum norm}

\author[D.\ Garc\'{\i}a]{Domingo Garc\'{\i}a}
\address[Domingo Garc\'{\i}a]{Departamento de An\'{a}lisis Matem\'{a}tico,
Universidad de Valencia, Doctor Moliner 50, 46100 Burjasot (Valencia), Spain.
\href{http://orcid.org/0000-0002-2193-3497}{ORCID: \texttt{0000-0002-2193-3497}}}
\email{domingo.garcia@uv.es}

\author[M.\ Maestre]{Manuel Maestre}
\address[Manuel Maestre]{Departamento de An\'{a}lisis Matem\'{a}tico,
Universidad de Valencia, Doctor Moliner 50, 46100 Burjasot
(Valencia), Spain.
\href{http://orcid.org/0000-0001-5291-6705}{ORCID: \texttt{0000-0001-5291-6705}}}
\email{manuel.maestre@uv.es}

\author[M.\ Mart\'{\i}n]{Miguel Mart\'{\i}n}
\address[Miguel Mart\'{\i}n]{Departamento de An\'{a}lisis Matem\'{a}tico e Instituto de Matemáticas (IMAG), Universidad de Granada, 18071 Granada, Spain.
\href{http://orcid.org/0000-0003-4502-798X}{ORCID: \texttt{0000-0003-4502-798X}}}
\email{mmartins@ugr.es}

\author[Rold\'an]{\'Oscar Rold\'an}
\address[Rold\'an]{Department of Mathematics Education, Dongguk University, 04620 Seoul, Republic of Korea. 
\href{https://orcid.org/0000-0002-1966-1330}{ORCID: \texttt{0000-0002-1966-1330}}}
\email{\texttt{oscar.roldan@uv.es}}

\thanks{}

\date{\today}
\keywords{Minimum-attaining, Quasi minimum-attaining, Bishop-Phelps-Bollob\'as, Radon-Nikodym property}
\subjclass[2020]{46B04 (primary), 46B03, 46B20, 46B22, 46B25 (secondary)}

\begin{abstract}
We study the set $\ma(X,Y)$ of operators between Banach spaces $X$ and $Y$ that attain their minimum norm, and the set $\qma(X,Y)$ of operators that quasi attain their minimum norm. We characterize the Radon-Nikodym property in terms of operators that attain their minimum norm and obtain some related results about the density of the sets $\ma(X,Y)$ and $\qma(X,Y)$. We show that every infinite-dimensional Banach space $X$ has an isomorphic space $Y$ such that not every operator from $X$ to $Y$ quasi attains its minimum norm. We introduce and study Bishop-Phelps-Bollob\'as type properties for the minimum norm, including the ones already considered in the literature, and we exhibit a wide variety of results and examples, as well as exploring the relations between them.
\end{abstract}
\maketitle

\section{Introduction}\label{section:introduction}

Let $X,Y$ be Banach spaces over the field $\bbk=\bbr$ or $\bbc$. Let $\call(X, Y)$, $\calk(X,Y)$ and $\calf(X,Y)$ respectively denote the spaces of bounded and linear operators, of compact operators, and of finite-rank operators from $X$ to $Y$. We denote by $B_X$, $S_X$ and $X^*$ the closed unit ball of $X$, the unit sphere of $X$, and the topological dual of $X$, respectively. Classical Banach sequence spaces and their finite-dimensional versions will also use standard notations, such as $c_0$, $\ell_p$ ($1\leq p\leq \infty$), and $\ell_p^n$ ($1\leq p\leq \infty$, $n\in\bbn$).

\subsection{Background on norm-attaining operators}

An operator $T\in\call(X,Y)$ is said to \textit{attain its norm} if there is $x\in S_X$ such that $\|T(x)\|=\|T\|$. The set of norm-attaining operators from $X$ to $Y$ will be denoted by $\na(X,Y)$. A cornerstone in Functional Analysis is a result by R. C. James that shows that a Banach space $X$ is reflexive if and only if $\na(X, \bbk)=X^*$ (see for instance \cite[Corollary 3.131]{FHHMZ11}). In 1961, E. Bishop and R. R. Phelps showed that for every Banach space $X$, the set $\na(X,\bbk)$ is dense in $X^*$ (see for instance \cite[Theorem 7.41]{FHHMZ11}), and they wondered whether this density also holds for operators from $X$ to $Y$. In 1963 J. Lindenstrauss answered that question in the negative in his seminal paper \cite{Lindenstrauss63}. He also showed that the density holds for some classes of Banach spaces, such as when $X$ is reflexive or when $Y$ has a geometrical property known as \textit{property $\beta$ of Lindenstrauss} (see Section \ref{section:bpbpm} for the definition). Since then, the study of when $\na(X,Y)$ is a dense subset of $\call(X,Y)$ has interested many researchers, and is still nowadays a very fruitful and active line of research in Functional Analysis.
We refer to \cite{Acosta06,CKL24} for surveys with the most important results on the density of norm-attaining operators. It is worth noting that it is still open nowadays whether $\na(X,Y)$ is always dense in $\call(X,Y)$ when $Y$ is finite-dimensional.

Recall that a Banach space $X$ has the Radon-Nikodym property (RNP) if every non-empty closed and bounded subset $X$ of $X$ is dentable (we refer to \cite[p. 217]{DU77} for several characterizations of this property). A Banach space $X$ is said to have \textit{property $A$ of Lindenstrauss} (respectively, \textit{property $B$ of Lindenstrauss}) if for every Banach space $Y$, $\na(X,Y)$ is dense in $\call(X,Y)$ (respectively, $\na(Y,X)$ is dense in $\call(Y,X)$). As a consequence of the work of J. Bourgain and R. E. Huff, it is known that a Banach space $X$ has the RNP if and only if every Banach space isomorphic to $X$ has property $A$ of Lindenstrauss (see \cite{Bourgain77,Huff80}). However, let us note that there are Banach spaces having property $A$ of Lindenstrauss but not the RNP (see \cite[Proposition 3.1]{Schachermayer83}), and there are Banach spaces with RNP but failing property $B$ of Lindenstrauss (see \cite[Appendix]{Gowers90}).

An operator $T\in\call(X,Y)$ is said to be \textit{quasi norm-attaining} if there is a sequence $\{x_n\}_n\subset S_X$ and a point $y\in \|T\|\cdot S_Y$ such that $T(x_n)$ converges to $y$. The class of quasi norm-attaining operators from $X$ to $Y$, $\qna(X,Y)$, was introduced and studied in \cite{CCJM22}. Analogously to properties $A$ and $B$ of Lindenstrauss, a Banach space has \textit{property quasi-$A$} (respectively, \textit{property quasi-$B$}) if for all Banach spaces $Y$, $\qma(X,Y)$ is dense in $\call(X,Y)$ (respectively, $\qma(Y,X)$ is dense in $\call(Y,X)$). Although the RNP does not imply property $B$ of Lindenstrauss, it has been shown in \cite[Corollary 3.8]{CCJM22} that a Banach space $X$ has the RNP if and only if every Banach space $Z$ isomorphic to $X$ has properties quasi-$A$ and quasi-$B$.

In 1970, B. Bollob\'as showed that for a Banach space $X$, not only the norm-attaining functionals are dense on $X^*$, but actually whenever a functional $x^*\in X^*$ almost attains its norm at a point $x\in S_X$, a nearby functional $y^*\in X^*$ attains its norm at a nearby point $y\in S_X$ (see for instance \cite[Exercise 7.53]{FHHMZ11}). This quantitative result for functionals is known nowadays as the Bishop-Phelps-Bollob\'as theorem. In 2008, M. D. Acosta, R. M. Aron, D. Garc\'{\i}a, and M. Maestre introduced a version of the Bishop-Phelps-Bollob\'as theorem for operators as follows.

\begin{definition}[{\cite[Definition 1.1]{AAGM08}}]
A pair of Banach spaces $(X,Y)$ is said to have the \textit{Bishop-Phelps-Bollob\'as property} (or just BPBp) if given $\varepsilon>0$ there is some $\eta=\eta(\varepsilon)>0$ such that whenever $T\in S_{\call(X,Y)}$ and $x\in S_X$ satisfy that $\|T(x)\|>1-\eta$, there exist $S\in S_{\call(X,Y)}$ and $y\in S_X$ such that $\|S(y)\|=1$, $\|S-T\|<\varepsilon$, and $\|x-y\|<\varepsilon$.
\end{definition}

It is known, for instance, that if $Y$ has property $\beta$ of Lindenstrauss (\cite[Theorem 2.2]{AAGM08}), if $X$ is uniformly convex (\cite[Theorem 3.1]{KL14}), or if both $X$ and $Y$ are finite-dimensional (\cite[Proposition 2.4]{AAGM08}), then the pair $(X,Y)$ has the BPBp. It is also known that there are pairs of Banach spaces $X$ and $Y$ such that $(X,Y)$ fails the BPBp despite the fact that $\na(X,Y)$ is dense in $\call(X,Y)$ (see \cite[Proposition 3.9 and Theorem 4.1]{AAGM08} and \cite[Theorem 7]{Bourgain77} for instance). Several versions of the BPBp have been considered in the literature. We will recall some of these versions, as they will be mentioned in Section \ref{section:bpbpm}.

\begin{remark}
A pair of Banach spaces $(X, Y)$ has a Bishop-Phelps-Bollob\'as type property if for every $\varepsilon>0$, $x\in S_X$, and $T\in S_{\call(X, Y)}$, there is some $\eta(\varepsilon, x, T)>0$ such that if the condition $\|T(x)\|\geq 1-\eta$ is met, then there exist $S\in S_{\call(X, Y)}$ and $y\in S_X$ such that $\|S(y)\|=1$, $\|S-T\|<\varepsilon$, and $\|x-y\|<\varepsilon$. In particular,
\begin{itemize}
\item If $\eta=\eta(\varepsilon)$ does not depend on $x$ and $T$, then $(X, Y)$ has the BPBp.
\item If $\eta=\eta(\varepsilon, T)$ depends on $T$ but not on $x$, then $(X, Y)$ has the \textit{\textbf{L}$_o$ property}.
\item If $\eta=\eta(\varepsilon, T)$ depends on $T$ but not on $x$, and also $S=T$, then $(X, Y)$ has the \textit{\textbf{L}$_{o,o}$ property}.
\item If $\eta=\eta(\varepsilon)$ does not depend on $x$ and $T$, and also $S=T$, then $(X, Y)$ has the \textit{BPBop}.
\item If $\eta=\eta(\varepsilon)$ does not depend on $x$ and $T$, and also $y=x$, then $(X, Y)$ has the \textit{BPBpp}.
\end{itemize}
\end{remark}

These and more versions of the BPBp have been widely studied in recent years. We refer to the surveys \cite{Acosta19,DGMR22} for a complete exposition of results about the BPBp and its versions up to 2022. In particular, \cite[Sections 4 and 5]{DGMR22} summarize the known results and relations between the versions of the BPBp mentioned above, among others.

\subsection{Background on minimum-attaining operators}

Given $T\in\call(X, Y)$, we define its \textit{minimum norm} as the value
$$m(T):=\inf\{\|T(x)\|:\, x\in S_X\}.$$
We say that $T\in\call(X, Y)$ \textit{attains its minimum norm} if there is some $x\in S_X$ such that $m(T)=\|T(x)\|$. The class of all minimum-attaining operators from $X$ to $Y$ is denoted by $\ma(X, Y)$. An operator $T\in\call(X, Y)$ such that $m(T)>0$ is called \textit{bounded below}, and the set of bounded below operators from $X$ to $Y$ is denoted by $\bl(X, Y)$. It is well known that an operator is bounded below if and only if it is a monomorphism, or equivalently, it is injective and has closed range (see for instance \cite[Section 10.2.3]{Kadets18}). If $X$ and $Y$ are linearly isomorphic Banach spaces, the set of linear isomorphisms from $X$ onto $Y$ is denoted by $\iso(X, Y)$. Note the following clear facts:
\begin{itemize}
\item If an operator $T\in\call(X, Y)$ is non-injective, then $m(T)=0$.
\item If $T\in\bl(X, Y)$, then $T\in\iso(X, T(X))$.
\end{itemize}

The study of minimum-attaining operators can be traced back to 1962 (see \cite{GT62}), but it remained overlooked for a long time afterwards. Later, in 2014, X. Carvajal and W. Neves re-initiated it (see \cite{CN14}), and since then, several works have been done about minimum-attaining operators, a topic that has proven to be useful due to its connections with the spectral theory of operators between complex Hilbert spaces (see for instance \cite{BR20a} and the references therein). Analogously to the theory of norm-attaining operators, in recent years there has been an increasing interest in studying when an operator is minimum-attaining, and when the set $\ma(X, Y)$ is dense in $\call(X, Y)$ (see \cite{BR21a, CN14, Chak20, Chak21, CL23, KR18}). We summarize now some background on the topic. 

Note first the following clear facts from \cite{CN14}:
\begin{itemize}
\item If $T\in\call(X, Y)$ is not injective, then $m(T)=0$ and $T\in\ma(X, Y)$.
\item If $T\in\call(X, Y)$ is injective and $X$ is infinite-dimensional, if $m(T)=0$, then $T\notin \ma(X, Y)$, and if $m(T)>0$, both scenarios can happen. 
\item If $X$ is finite-dimensional or $Y$ is finite-dimensional, then $T\in\ma(X, Y)$ (note that the converse is false, see \cite[Remark 2.6]{Chak20}). 
\item In particular, in general $\calf(X, Y)\subset \ma(X, Y)$, so the analogous of the Bishop-Phelps theorem on density for $m$ always holds in every Banach space $X$ (see also \cite[Proposition 2.8]{CL23} for a stronger claim). 
\item If $X$ is infinite-dimensional and $T\in\calk(X, Y)$, then $m(T)=0$.
\end{itemize}

The class $\qma(X,Y)$ of \textit{quasi minimum-attaining operators} from $X$ to $Y$ was introduced in \cite[Definition 4.1]{Chak20} by U. S. Chakraborty as the class of those operators $T\in\call(X,Y)$ such that there is a sequence $\{x_n\}_n\subset S_X$ and a point $y\in m(T)\cdot S_Y$ such that $T(x_n)$ converges to $y$. It is shown in \cite[Section 4]{Chak20} that $\ma(X, Y)\subset \qma(X, Y)$, and $\call(X, Y)\backslash \bl(X, Y)\subset \qma(X,Y)$, and the inclusions cannot be reversed in general. Moreover, if $T\in\call(X,Y)$ has closed range (in particular if $T\in\bl(X,Y)$), then $T\in\ma(X,Y)$ if and only if $T\in\qma(X,Y)$, and there are operators that are not quasi minimum-attaining (see \cite[Example 4.8]{Chak20}). Chakraborty also showed in that paper that all the following claims are equivalent for Banach spaces: (a)\, $X$ is finite-dimensional, (b)\, $\qma(X, Y)=\ma(X, Y)$ for all $Y$, (c)\, $\ma(X, Y)=\call(X, Y)$ for all $Y$, (d)\, $\qma(X, Y)=\call(X, Y)$ for all $Y$, (e)\, $\qma(Y, X)=\ma(Y, X)$ for all $Y$, (f)\, $\ma(Y, X)=\call(Y, X)$ for all $Y$. 

The density of $\ma(X,Y)$ in $\call(X,Y)$ was first studied in \cite{KR18} and \cite{Chak20}. We say that a Banach space $X$ has \textit{property $A_M$} (respectively, \textit{property $B_M$}) if for all Banach spaces $Y$, $\ma(X,Y)$ is dense in $\call(X,Y)$ (respectively, $\ma(Y,X)$ is dense in $\call(Y,X)$). The following is known. Let $X$ and $Y$ denote Banach spaces as usual.
\begin{enumerate}
\item If $X$ and $Y$ are infinite-dimensional complex Hilbert spaces, then $\ma(X,Y)$ is dense in $\call(X,Y)$ (see \cite[Corollary 3.7]{KR18}). 
\item $\ma(X, Y)$ and $\qma(X, Y)$ are not always dense (see \cite[Example 3.4 and Remark 4.10]{Chak20}). 
\item While there are Banach spaces such that $\qna(X,Y)$ is dense in $\call(X,Y)$ but $\na(X,Y)$ is not (see \cite[Example 3.7]{CCJM22}), for the minimum counterpart this can never be the case, since $\overline{\ma(X, Y)}$ and $\overline{\qma(X, Y)}$ always coincide (see \cite[Theorem 4.11]{Chak20}).
\item If every $Z$ isomorphic to $X$ has property $A$ of Lindenstrauss (that is, if $X$ has RNP), then $X$ has property $A_M$. Also, if every closed subspace $Z$ of $Y$ has property $A$ of Lindenstrauss (for instance, if $Y$ has RNP), then $Y$ has property $B_M$ (see \cite[Theorem 3.8]{Chak20}).
\item If $Y$ is a separable Banach space with property $B_M$, then $B_Y$ has an extreme point (see \cite[Theorem 3.12]{Chak20}).
\end{enumerate}

Finally, some versions of the BPBp for the minimum norm have been considered in the literature (see \cite{BR21a,Chak20,Chak21}). In particular, the following has been shown (we will omit some details here, as this will be discussed throughout Section \ref{section:bpbpm}).
\begin{enumerate}
\item If $H$ is a complex Hilbert space, $(H, H)$ satisfies an analogous property to the $\Lo$ for $m$ (see \cite[Theorems 3.5 and 3.8]{BR21a}).
\item Chakraborty introduced in \cite[Definition 1.4]{Chak21} an analogous property to the $\Loo$ for $m$, the AMp, and he showed for instance that the following claims are equivalent for real Banach spaces: (a)\, $X$ is finite-dimensional, (b)\, $(X, Y)$ has the $\amp$ for all $Y$, (c)\, $(Y, X)$ has the $\amp$ for all $Y$.
\item Finally, let us notice that Chakraborty showed in \cite[Lemma 3.7]{Chak20} that operators $T\in\call(X,Y)\backslash \bl(X,Y)$ always satisfy a property analogous to the BPBpp-$m$. Let us explicitly state this result.

\textit{Let $X$ and $Y$ be Banach spaces and $\varepsilon>0$. If $T\in\call(X, Y)$ is such that $m(T)=0$ and $x\in S_X$ is such that $\|T(x)\|<\varepsilon$, then there is $S\in\call(X, Y)$ such that $\|S(x)\|=m(S)=0$ and $\|T-S\|<\varepsilon$.}
\end{enumerate}

\subsection{Outline of the document}\label{subsection:structure}

The rest of the document is structured as follows. In Section \ref{section:density-MA-QMA}, we provide new results related to the density of the sets $\ma(X,Y)$ and $\qma(X,Y)$. We improve some results of \cite{Chak20} and get a characterization of the Radon-Nikodym property in terms of operators that attain their minimum (see Theorem \ref{Thm:Char-RNP-AM-BM}). We also show that every infinite-dimensional Banach space $X$ has an isomorphic space $Y$ such that $\qma(X,Y)\neq \call(X,Y)$ (see Theorem \ref{theo:qma-not-l-isom-domains}).

In Section \ref{section:bpbpm}, we study Bishop-Phelps-Bollob\'as-type properties for the minimum norm. We consider the versions from the literature \cite{BR21a,Chak20,Chak21} and another natural version that is partially related to the usual BPBp (see Theorem \ref{theo:bpbp-then-bpbpm}), and we provide a wide list of classes of Banach spaces that do or that do not satisfy all these properties. We also study the relations between those properties. Among other results, we show that many pairs of classical Banach sequence spaces satisfy some of these properties (see for instance Examples \ref{exs:BPBpp-m} and Theorem \ref{theo:bpbp-then-bpbpm} and its consequences), and we also improve the claim from \cite[Theorems 3.5 and 3.8]{BR21a} by showing that a much wider class of Banach spaces satisfy a strenghtening of the property stated in those results (see Corollary \ref{cor:Y-un-conv}).

\section{New results on minimum-attaining operators}\label{section:density-MA-QMA}

In \cite[Corollary 3.5]{CCJM22}, a characterization of the Radon-Nikodym property was provided in terms of operators that quasi attain their norm. We will present in Theorem \ref{Thm:Char-RNP-AM-BM} a new characterization of the RNP in terms of operators that attain their minimum. 

We say that a Banach space $X$ has \textit{property quasi-$A_M$} (respectively, \textit{property quasi-$B_M$}) if for all Banach spaces $Y$, $\qma(X,Y)$ is dense in $\call(X,Y)$ (respectively, $\qma(Y,X)$ is dense in $\call(Y,X)$. Let us begin by noting the following: it is known that properties $B$ and quasi-$B$ are not equivalent, since the RNP implies property quasi-$B$ (see \cite[Corollary 3.5]{CCJM22}), but it does not imply property $B$ in general (see \cite[Appendix]{Gowers90}). As for properties $A$ and quasi-$A$, it is not known whether or not they are equivalent (see \cite[Problem 7.7]{CCJM22}). However, for the case of minimum attaining operators, it was shown in \cite[Theorem 4.11]{Chak20} that for all Banach spaces $X$ and $Y$, $\overline{\qma(X, Y)}=\overline{\ma(X, Y)}$, and so, the following clearly holds. 

\begin{proposition}\label{prop:quasi-M-is-M}
Let $X$ be a Banach space. 
\begin{enumerate}
\item If $X$ has property quasi-$A_M$, then it has property $A_M$.
\item If $X$ has property quasi-$B_M$, then it has property $B_M$.
\end{enumerate}
\end{proposition}

Let $X$ be a Banach space. In \cite[Theorem 3.8]{Chak20} it was shown that if every space $Z$ isomorphic to $X$ has property $A$, then $X$ has property $A_M$, and that if every closed subspace $Z$ of $X$ has property $A$, then $X$ has property $B_M$. In view of \cite[Lemma 2.1]{CCJM22}, one can equivalently state the first claim as follows.

\begin{remark}
If $X$ is a Banach space such that every Banach space isomorphic to $X$ has property quasi-$A$, then $X$ has property $A_M$.
\end{remark}

Moreover, the second part of \cite[Theorem 3.8]{Chak20} can also be adapted accordingly.

\begin{proposition}\label{prop:QuasiA-Equiv-Norms-AM}
If $Y$ is a Banach space such that every closed subspace $Z$ of $Y$ has property quasi-$A$, then $Y$ has property $B_M$.
\end{proposition}

\begin{proof} Let $X$ be a Banach space and let $T\in\call(X, Y)$. Note that if $T\notin \bl(X, Y)$, then $T\in\overline{\ma(X, Y)}$ by \cite[Lemma 3.7]{Chak20}, so that case is finished. Assume now that $T\in \bl(X, Y)$. Then $T:X\rightarrow T(X)$ is an isomorphism. Since $T(X)$ has property quasi-$A$, $\qna(T(X), X)$ is dense in $\call(T(X), X)$. Since $T^{-1}\in\iso(T(X), X)$ and $\iso(T(X), X)$ is open in $\call(T(X), X)$, there exists a sequence of operators $(S_n)_n$ in $\iso(T(X), X)\cap\qna(T(X), X)$ such that $\|S_n-T^{-1}\|\rightarrow 0$. However, by \cite[Lemma 2.1]{CCJM22}, for all $n\in\bbn$, $S_n\in \iso(T(X), X)\cap \na(T(X), X)$. Also, by \cite[Lemma 3.2.(c)]{Chak20}, if for any $n$, $S_n$ attains its norm at some point $y_n\in S_{T(X)}$, then $S_n^{-1}:X\rightarrow T(X)\subset Y$ attains its minimum norm at $\frac{S_n(y_n)}{\|S_n(y_n)\|}\in S_X$, and $\|S_n^{-1}-T\|\rightarrow 0$. Consequently, $T\in\overline{\ma(X, Y)}$, and so, $Y$ has property $B_M$ as desired.
\end{proof}

Using similar ideas, the following can also be shown. The proof is very similar, but we include the details for the sake of completeness.

\begin{proposition}\label{Prop:Quasi-B-AM}
If $X$ is a Banach space with property quasi-$B$, then $X$ has property $A_M$.
\end{proposition}

\begin{proof}
Let $X$ be a Banach space with property quasi-$B$, and let $Y$ be a Banach space. Let $T\in\call(X, Y)$. Note that if $T\notin \bl(X, Y)$, then $T\in\overline{\ma(X, Y)}$ by \cite[Lemma 3.7]{Chak20}, so that case is finished. Assume now that $T\in \bl(X, Y)$. Then $T:X\rightarrow T(X)$ is an isomorphism. Since $X$ has property B, $\na(T(X), X)$ is dense in $\call(T(X), X)$. Since $T^{-1}\in\iso(T(X), X)$ and $\iso(T(X), X)$ is open in $\call(T(X), X)$, there exists a sequence of operators $(S_n)_n$ in $\iso(T(X), X)\cap\qna(T(X), X)$ such that $\|S_n-T^{-1}\|\rightarrow 0$. However, by \cite[Lemma 2.1]{CCJM22}, for all $n\in\bbn$, $S_n\in \iso(T(X), X)\cap \na(T(X), X)$. Also, by \cite[Lemma 3.2.(c)]{Chak20}, if for any $n$, $S_n$ attains its norm at some $y_n\in S_{T(X)}$, then $S_n^{-1}:X\rightarrow T(X)\subset Y$ attains its minimum norm at $\frac{S_n(y_n)}{\|S_n(y_n)\|}\in S_X$. Consequently, $T\in\overline{\ma(X, Y)}$, and so, $X$ has property $A_M$ as desired.
\end{proof}

In particular, property $B$ of Lindenstrauss implies property $A_M$, so by \cite[Theorem 1]{Partington82}), we get the following consequence.

\begin{corollary}\label{cor:AM-renorming}
Let $X$ be a Banach space. Then there exists a Banach space $Z$ isomorphic to $X$ such that $Z$ has property $A_M$.
\end{corollary}

Let us remark some facts. The space $\ell_2$ has the RNP, so it has property $A$ of Lindenstrauss as well as properties $A_M$ and $B_M$ (see \cite[Corollary 3.10]{Chak20}), but it does not have property $B$ of Lindenstrauss (see \cite[Appendix]{Gowers90}). On the other hand, $c_0$ has property $B$ of Lindenstrauss (and hence property $A_M$ by Proposition \ref{Prop:Quasi-B-AM}), but it does not satisfy property $A$ of Lindenstrauss or property $B_M$ (see \cite{Lindenstrauss63} and \cite[Example 3.4]{Chak20}). Whether property $A$ of Lindenstrauss implies properties $A_M$ or $B_M$, or property $B_M$ implies property $A$ of Lindenstrauss, remains open (see also \cite[Question 3.11]{Chak20}).

Recall that in \cite[Corollary 3.10]{Chak20} it was shown that the RNP implies properties $A_M$ and $B_M$. In view of \cite{Huff80}, we can characterize now the RNP in terms of minimum-attaining operators.

\begin{theorem}\label{Thm:Char-RNP-AM-BM}
Let $X$ be a Banach space. The following claims are equivalent.
\begin{enumerate}
\item $X$ has the RNP.
\item Every Banach space isomorphic to $X$ has property $A_M$.
\item Every Banach space isomorphic to $X$ has property $B_M$.
\end{enumerate}
\end{theorem}

\begin{proof}
Note that the RNP implies both properties $A_M$ and $B_M$ (see \cite[Corollary 3.10]{Chak20}), and it is stable under equivalent renormings, so the implications $(1)\Rightarrow (2)$ and $(1)\Rightarrow (3)$ follow. 


Finally, we show by contradiction that the other two implications hold. Indeed, if $X$ does not have the RNP, then by \cite{Huff80}, there exist two equivalent norms on $X$, $\|\cdot\|_a$ and $\|\cdot\|_b$, such that the identity 
$$\id:(X, \|\cdot\|_a)\rightarrow (X, \|\cdot\|_b)$$
cannot be approximated by norm-attaining operators. In particular, since $\id\in\iso((X, \|\cdot\|_a), (X, \|\cdot\|_b))$, its inverse cannot be approximated by minimum-attaining operators. Therefore, $X$ can have neither property $A_M$ for every equivalent norm nor property $B_M$ for every equivalent norm.
\end{proof}

In \cite[Theorem 3.12]{Chak20}, it was shown that if a separable infinite-dimensional Banach space $Y$ has property $B_M$, then its unit ball has an extreme point. We show now an improvement of this result. Recall that a Banach space $X$ is said to be \textit{locally uniformly rotund} (or just LUR) if 
$$\lim_k \|x_k-x\|=0\quad \text{whenever}\quad \lim_k \left\|\frac{1}{2} (x_k+x)\right\|=\lim_k \|x_k\|=\|x\|.$$
It is known that every separable Banach space admits a LUR equivalent norm (see \cite[Theorem 1]{Troyanski71}). For a Banach space $Y$, a point $y\in B_Y$ is a \textit{strongly exposed point} if there is $y^*\in S_{Y^*}$ such that whenever a sequence $\{y_n\}_n\subset S_X$ satisfies that $\lim_n y^*(y_n)=1$, then $y_n$ converges to $y_0$ (in particular, $y^*(y_0)=1$). Note that every strongly exposed point is an extreme point. The following holds.

\begin{proposition}\label{prop:LUR+BM-SEpoint}
Let $Y$ be a Banach space with property $B_M$. If $Y$ admits a LUR equivalent norm (in particular, if $Y$ is separable), then its unit ball has a strongly exposed point.
\end{proposition}

\begin{proof}
Suppose that $Y$ has no strongly exposed points. Let $Y_0$ denote the space $Y$ equipped with a LUR equivalent norm. If $T:Y_0\rightarrow Y$ is an isomorphism, then $T^{-1}:Y\rightarrow Y_0$ can not attain its norm, as the unit ball of $Y$ has no strongly exposed points (note that, by \cite[Lemma 2.2.(1)]{JMR23}, if $T^{-1}$ attained its norm, it would attain it at a strongly exposed point). Since $m(T)=\frac{1}{\|T^{-1}\|}$, we get that $T$ can not attain $m(T)$. But since $\iso(Y_0, Y)$ is an open set, $T$ can not be in $\overline{\ma(Y_0, Y)}$. Hence, $Y$ does not satisfy property $B_M$, which is a contradiction.
\end{proof}

\begin{remark}
Note that the previous proposition allows us to get with a different approach the implication $(3)\Rightarrow (1)$ of Theorem \ref{Thm:Char-RNP-AM-BM} for separable Banach spaces, since if every Banach space isomorphic to $X$ has a strongly exposed point, then $X$ has the RNP (see \cite[Corollary 1]{Huff80}, and note that strongly exposed points are always denting).
\end{remark}

In \cite[Theorem 4.17]{Chak20} it was shown that if $X$ is any infinite-dimensional Banach space, then there always exists a Banach space $Y$ such that $\qma(X,Y)\neq \call(X,Y)$. This was achieved by considering an infinite-dimensional separable subspace $X_0$ of $X$ and setting $Y=X\oplus_1 (\ell_\infty \oplus_1 X/X_0)$. To finish this section we will show that the claim from that theorem can also be always achieved with some range space $Y$ that is isomorphic to the original space $X$. We do not know, however, if for every infinite-dimensional space $Y$, there is always a Banach space $X$ such that $\qma(X,Y)\neq \call(X,Y)$.

We will use a result on remotality. Given a Banach space $X$, a set $S\subset X$ is said to be \textit{remotal} from a point $x\in X$ if there exists $y\in S$ such that $\|x-y\|=\sup\{\|x-z\|:\, z\in S\}$. There have been several works investigating the existence of non-remotal sets with nice properties (such as being closed and convex) inside the unit ball of a Banach space $X$ (see for instance \cite{MR10,Vesely09}). In \cite[Proposition 2.1]{Vesely09} it was shown that given a \textit{real} infinite-dimensional Banach space $X$ and any $\varepsilon\in (0,1)$, there is always a bounded closed and absolutely convex set $D$ with $(1-\varepsilon) B_X\subset D\subset B_X$ and such that $D$ is not remotal from $0$. The argument used in the proof only works for the real case. However, this result is also true in the complex case.

\begin{lemma}\label{lemma:remotality}
Let $X$ be an infinite-dimensional Banach space, and let $\varepsilon\in (0,1)$. Then there is a bounded closed and absolutely convex set $D$ with $(1-\varepsilon) B_X\subset D\subset B_X$ and such that $D$ is not remotal from $0$.
\end{lemma}

Given a set $S\subset X$, let $\aconv(S)$ and $\overline{\aconv}(S)$ respectively denote the absolutely convex hull and the closed absolutely convex hull of $S$.

\begin{proof}[Proof of Lemma \ref{lemma:remotality}]
In \cite[Corollary 4.6]{CCJM22} it was shown that there is a subset $K$ of $B_X$ that is bounded, closed, absolutely convex, and not remotal from $0$. In fact, we can have it so that $\sup\{\|x\|:\, x\in K\}=1$ (but $\|x\|<1$ for all $x\in K$). Note that $K=\overline{\aconv}(K)$. Let $D_1:=\aconv(K\cup (1-\varepsilon) B_X)$. This set is absolutely convex, and it also satisfies that $\sup\{\|x\|:\, x\in D_1\}=1$ but $\|x\|<1$ for all $x\in D_1$, since for all $x\in D_1$, there are $y\in K$ and $z\in (1-\varepsilon)B_X$ such that $x\in\conv\{y,z\}$, and note that $\|y\|,\|z\|<1$. Let $D=\overline{D_1}$, and note that $(1-\varepsilon)B_X\subset D\subset B_X$. We show that $D$ is also not remotal from $0$ by contradiction. Suppose otherwise that there exists $x\in D$ such that $\|x\|=1$. Then, there is a sequence $\{x_n\}_n\subset D_1$ such that $x=\lim_n x_n$. For each $n\in\bbn$, there are $y_n\in K$, $z_n\in(1-\varepsilon)B_X$, and $a_n,b_n\in B_\bbk$ such that $|a_n|+|b_n|\leq 1$ and $x_n = a_n y_n + b_n z_n$.

Note that there is a subsequence $\{|a_{n_k}|\}_k$ of $\{|a_{n}|\}_n$ converging to $1$, since otherwise we would have $\limsup_n |a_n| < 1$, and so, since $\varepsilon>0$, for each $n\in\bbn$ we have
$$\|x_n\|\leq |a_n| \|y_n\| + (1-\varepsilon)|b_n|\leq |a_n| + (1-\varepsilon) (1-|a_n|)=1+\varepsilon |a_n|-\varepsilon,$$
so $\limsup_n \|x_n\| < 1=\|x\|$, which would be a contradiction.

Therefore, $b_{n_k}$ converges to $0$, and so, we can write
$$x=\lim_k x_{n_k} = \lim_k a_{n_k} y_{n_k}\in \overline{\aconv}(K)=K,$$ 
and so, $\|x\|<1$, which is a contradiction. Therefore $D$ is not remotal from $0$.
\end{proof}

Using this lemma we can finally show the promised result.

\begin{theorem}\label{theo:qma-not-l-isom-domains}
Let $X$ be an infinite-dimensional Banach space. Then there is a Banach space $Y$ isomorphic to $X$ such that $\qma(X, Y)\neq\call(X, Y)$.
\end{theorem}

\begin{proof}
Fix any $\varepsilon>0$. By Lemma \ref{lemma:remotality}, there is a bounded, closed, and absolutely convex set $D$ with $(1-\varepsilon)B_X\subset D\subset B_X$ such that $D$ is not remotal from $0$. Using Minkowski's functional, there exists an equivalent norm on $X$, $\vertiii{\cdot}$, such that $B_{(X, \vertiii{\cdot})}=D$. Call $Y=(X, \vertiii{\cdot})$. Note that $\id:Y\rightarrow X$ is not norm-attaining, by the non-remotality from $0$ of $D$, so its inverse $\id:X\rightarrow Y$ is not minimum-attaining, and since it is an isomorphism, in particular it is not quasi minimum-attaining.
\end{proof}

\section{Bishop-Phelps-Bollob\'as type properties for the minimum}\label{section:bpbpm}

Several variations of the Bishop-Phelps-Bollob\'as property have been considered in the literature (see for instance \cite[Sections 4 and 5]{DGMR22} for a detailed exposition of several versions of the BPBp and the relations between them). The corresponding properties for the minimum norm can be defined in an analogous way. In this section we will consider some of these properties and we will study their implications and the relations between them, building upon the results from the literature.

We begin with a very strong version.

\begin{definition}
A pair of Banach spaces $(X,Y)$ has the \textit{Bishop-Phelps-Bollob\'as point property for the minimum} (abbreviated BPBpp-$m$) if for every $\varepsilon\in (0,1)$ and every $\rho\geq 0$, there exists $\eta(\varepsilon, \rho)\in (0,1)$ (converging to $0$ with $\varepsilon$) satisfying that whenever $T\in\call(X,Y)$ with $m(T)=\rho$ and $x\in S_X$ are such that $\|T(x)\|<\rho (1+\eta(\varepsilon,\rho))$, if $\rho>0$, or $\|T(x)\|<\eta(\varepsilon, \rho)$, if $\rho=0$, there exists $S\in\call(X, Y)$ such that $m(S)=\|S(x)\|=\rho$, and $\|S-T\|<\varepsilon$.
\end{definition}

U. S. Chakraborty showed in \cite[Lemma 3.7]{Chak20} that for the class of operators $T\in\call(X, Y)$ with $m(T)=0$, the previous definition is always satisfied with $\eta(\varepsilon, 0)=\varepsilon$. As a consequence, it suffices to state the property as follows.

\begin{remark}
A pair of Banach spaces $(X,Y)$ has the BPBpp-$m$ if and only if for every $\varepsilon\in (0,1)$, there exists $\eta(\varepsilon)\in (0,1)$ satisfying that whenever $T\in \call(X, Y)$ with $m(T)=1$ and $y\in S_X$ are such that $\|T(x)\|<1+\eta(\varepsilon)$, there exists $S\in\call(X, Y)$ such that $m(S)=\|S(x)\|=1$, and $\|S-T\|<\varepsilon$. We omit the details, see Proposition \ref{remark:equiv-def} for an analogous equivalence on a weaker property.
\end{remark}

In particular, \cite[Lemma 3.7]{Chak20} yields the following.

\begin{proposition}
If the Banach spaces $X$ and $Y$ are such that $\bl(X, Y)=\emptyset$, then $(X,Y)$ satisfies the BPBpp-$m$. 
\end{proposition}

As a consequence, the following strong version of the Bishop-Phelps-Bollob\'as theorem is always true for the minimum norm.

\begin{corollary}
If $X$ is any Banach space, then $(X,\bbk)$ satisfies the BPBpp-$m$.
\end{corollary}

Let us see some more examples.

\begin{example}\label{exs:BPBpp-m}
Given two Banach spaces $X$ and $Y$, recall that an operator $T\in\call(X,Y)$ is called \textit{strictly singular} if $T$ is not bounded below on any closed infinite-dimensional subspace of $X$. If $\sss(X,Y)$ denotes the class of strictly singular operators from $X$ to $Y$, it is clear and well known that $\calk(X,Y)\subset \sss(X, Y)\subset \call(X,Y)\backslash \bl(X,Y)$. We have the following.
\begin{enumerate}
\item If $\sss(X,Y)=\call(X,Y)$ (in particular, if $\calk(X,Y)=\call(X,Y)$), then $(X,Y)$ has the BPBpp-$m$. A wide collection of pairs of spaces satisfying this property can be found in \cite[Theorem 1]{KY17}. In particular, this holds for instance if $X=c_0$ and $Y=\ell_p$ ($1\leq p<\infty$), if $X=\ell_p$ ($1\leq p\leq \infty$) and $Y=c_0$, or if $X=\ell_p$ ($1<p\leq \infty$) and $Y=\ell_q$ ($1<q<\infty$, $p\neq q$). Note that it is open whether the pair $(c_0, \ell_1)$ has the BPBp in the real case, but for the minimum counterpart, we have that both pairs $(c_0, \ell_1)$ and $(\ell_1, c_0)$ satisfy the BPBpp-$m$.
\item Note also that there exist pairs of infinite-dimensional Banach spaces $(X,Y)$ with the BPBpp-$m$ but such that not all operators are strictly singular. Indeed, it suffices to take a Banach space $X$ with a $1$-complemented proper infinite-dimensional subspace $Y$ such that the cardinality of the Hamel basis of $X$ is bigger than that of $Y$ (so that no operator is injective), and note that the corresponding projection is bounded below on $Y$ (see for instance Theorem \ref{theorem:ma-not-loom} for an example of such pair of spaces). 
\item More generally, any pair of Banach spaces $(X,Y)$ such that no operator in $\call(X,Y)$ is injective satisfies the BPBpp-$m$. Note that this is sometimes possible even with spaces such that the corresponding Hamel basis have the same cardinality: let $\Gamma$ be an index set with cardinality bigger than the continuum $c$, and note that if $p,q\in (1,\infty)$ are such that $\frac{1}{p}+\frac{1}{q}=1$, then every operator from $\ell_p(\Gamma)$ to $\ell_q(\Gamma)$ is compact by Pitt's theorem, and so, it is non-injective, as its range is separable.
\item There exist pairs of Banach spaces $(X,Y)$ with the BPBpp-$m$ but such that not every operator has minimum $0$. Indeed, this is trivially the case for $X=Y=\bbr$.
\end{enumerate}
\end{example}

We turn now our attention to local versions of the BPBp for the minimum. U. S. Chakraborty introduced in \cite[Definition 1.4]{Chak21} the \textit{aproximate minimizing property} (AMp) as the minimum counterpart of the \textbf{L}$_{o,o}$ property. Although this was done in the real case, the definition is analogous in the complex case. Let us denote this property \textbf{L}$_{o,o}$-$m$ in this paper.

\begin{definition}
A pair of Banach spaces $(X,Y)$ has the \textit{\textbf{L}$_{o,o}$ property for the minimum} (abbreviated \textbf{L}$_{o,o}$-$m$) if for every $\varepsilon\in (0,1)$ and every $T\in \call(X,Y)$, there exists $\eta(\varepsilon, T)\in (0,1)$ (converging to $0$ with $\varepsilon$) satisfying that whenever $x\in S_X$ is such that $\|T(x)\|<m(T) (1+\eta(\varepsilon,T))$, if $m(T)>0$, or $\|T(x)\|<\eta(\varepsilon,T)$, if $m(T)=0$, there exists $y\in S_X$ such that $\|T(y)\|=m(T)$, and $\|x-y\|<\varepsilon$.
\end{definition}

\begin{remark}
Note that this definition is trivially equivalent to the one given in \cite[Definition 1.4]{Chak21}, since, once more, in view of \cite[Lemma 3.7]{Chak20}, the condition ``$\|T(x)\|<\eta(\varepsilon,T)$ if $m(T)=0$'' can be removed from the definition by taking $\eta(\varepsilon, T)\leq \varepsilon$ for these operators. 
\end{remark}

The following result has been stablished in \cite[Section 3]{Chak20} in the real case, although a quick glance at the proof shows that the result also holds in the complex case.

\begin{theorem}[{\cite[Theorems 3.5 and 3.6]{Chak20}}]
Let $X$ be a Banach space. The following claims are equivalent:
\begin{itemize}
\item[a)] $X$ is finite-dimensional.
\item[b)] $(X,Y)$ has the \textbf{L}$_{o,o}$-$m$ for all Banach spaces $Y$.
\item[c)] $(Y,X)$ has the \textbf{L}$_{o,o}$-$m$ for all Banach spaces $Y$.
\end{itemize}
\end{theorem}

Note that if $(X,Y)$ has the \textbf{L}$_{o,o}$-$m$, then necessarily we must have $\ma(X, Y)=\call(X,Y)$. Chakraborty found in \cite[Example 3.7]{Chak21} a pair of Banach spaces such that $\ma(X,Y)\neq \call(X,Y)$, but $(X,Y)$ still fails the \textbf{L}$_{o,o}$-$m$ for the class of minimum-attaining operators (that is, the \textbf{L}$_{o,o}$-$m$ but only defined for those operators $T$ that are in $\ma(X,Y)$). On the other hand, we show now that the condition $\ma(X, Y)=\call(X,Y)$ is not enough to get the \textbf{L}$_{o,o}$-$m$. 

\begin{theorem}\label{theorem:ma-not-loom}
There exist infinite-dimensional Banach spaces $X$ and $Y$ such that $\ma(X,Y)=\call(X,Y)$ but the pair $(X,Y)$ fails the \textbf{L}$_{o,o}$-$m$.
\end{theorem}

Recall that for an index set $\Gamma$, the Banach space $c_0(\Gamma)$ is defined as $\{f:\Gamma\rightarrow \bbk:\, \text{for all }\varepsilon>0,\{\gamma\in \Gamma:\, |f(\gamma)|>\varepsilon\}\text{ is finite}\}$, endowed with the supremum norm. 

\begin{proof}[Proof of Theorem \ref{theorem:ma-not-loom}]
Let $\Gamma$ be a big enough index set so that the Hamel basis of $c_0(\Gamma)$ has cardinality bigger than that of $c_0$, and note that $\ma(c_0(\Gamma), c_0)=\call(c_0(\Gamma), c_0)$, as no operator can be injective. Fix a countable set $\{\beta_n\}_n\subset \Gamma$. Define $T:c_0(\Gamma)\rightarrow c_0$ as follows: for each $f\in c_0(\Gamma)$, $T(f)$ is the sequence $(T(f)_n)_n\in c_0$ given by $T(f)_n:=\frac{f(\beta_n)}{n}$ for each $n\in\bbn$.
Note that $\|T\|=1$. Moreover, we have that
$$m(T)\leq \inf_n \|T(\delta_{\beta_n})\| = \inf_n \frac{1}{n} = 0,$$
where $\delta_{\beta_n}$ is the Dirac delta function at $\beta_n$. By definition, it is also clear that for all $f\in c_0(\Gamma)$, $\|T(f)\|=0$ if and only if $f(\beta_n)=0$ for all $n\in\bbn$.

Now, fix $\varepsilon\in (0,1)$ and suppose that there exists $\eta(\varepsilon, T)\in (0,1)$ for which the \textbf{L}$_{o,o}$-$m$ holds. There must exist some $n_0\in\bbn$ such that $\eta(\varepsilon, T)>\frac{1}{n_0}$. Fix $n>n_0$. We have that $\|T(\delta_{\beta_n})\|<\frac{1}{n}<\eta(\varepsilon, T)$, so there must exist $g\in S_{c_0(\Gamma)}$ such that $\|T(g)\|=0$ and $\|g-\delta_n\|<\varepsilon$, but that is a contradiction. Therefore, $(c_0(\Gamma), c_0)$ fails the \textbf{L}$_{o,o}$-$m$.
\end{proof}

\begin{corollary}
There are pairs of Banach spaces $(X,Y)$ satisfying the BPBpp-$m$ but not the \textbf{L}$_{o,o}$-$m$.
\end{corollary}

We will also see later in Theorem \ref{theorem:ell22,ell12} that the \textbf{L}$_{o,o}$-$m$ does not imply the BPBpp-$m$. We turn now our attention to a less restrictive property, the minimum analogous to the \textbf{L}$_o$.

\begin{definition}
A pair of Banach spaces $(X,Y)$ has the \textit{\textbf{L}$_{o}$ property for the minimum} (abbreviated \textbf{L}$_{o}$-$m$) if for every $\varepsilon\in (0,1)$ and every $T\in \call(X,Y)$, there exists $\eta(\varepsilon, T)\in (0,1)$ (converging to $0$ with $\varepsilon$) satisfying that whenever $x\in S_X$ is such that $\|T(x)\|<m(T) (1+\eta(\varepsilon,T))$, if $m(T)>0$, or $\|T(x)\|<\eta(\varepsilon,T)$, if $m(T)=0$, there exist $S\in\call(X,Y)$ and $y\in S_X$ such that $\|S(y)\|=m(S)=m(T)$, and $\|x-y\|<\varepsilon$.
\end{definition}

\begin{remark}
Once more, in view of \cite[Lemma 3.7]{Chak20}, the condition ``$\|T(x)\|<\eta(\varepsilon,T)$ if $m(T)=0$'' can be removed from the definition by taking $\eta(\varepsilon, T)\leq \varepsilon$ for these operators.
\end{remark}

N. Bala and G. Ramesh showed in \cite[Theorems 3.5 and 3.8]{BR21a} that if $H$ is a complex Hilbert space, then $(H,H)$ satisfies the \textbf{L}$_o$-$m$. We will improve this result in Corollary \ref{cor:Y-un-conv} by showing that a much wider class of Banach spaces satisfies the following strengthening of this property, where instead of having $\eta$ depending on each operator, we let it depend on a class of operators in a more uniform way.

\begin{definition}
A pair of Banach spaces $(X,Y)$ is said to have the \textit{bounded Bishop-Phelps-Bollob\'as property for the minimum norm} (or just bounded BPBp-$m$) if given $\varepsilon>0$ and $0\leq \rho\leq R$ there exists some $\eta(\varepsilon, R, \rho)\in (0,1)$ (converging to $0$ with $\varepsilon$) such that whenever $x\in S_X$ and $T\in\mathcal{L}(X,Y)$ with $\|T\|\leq R$ and $m(T)= \rho$ satisfy $\|T(x)\|<m(T)(1+\eta(\varepsilon, R, \rho))$ if $\rho>0$, or $\|T(x)\|<\eta(\varepsilon, R, \rho)$ if $\rho=0$, there exist $S\in\mathcal{L}(X,Y)$ and $y\in S_X$ such that 
$$\|S(y)\|=m(S)=m(T)=\rho,\quad \|T-S\|<\varepsilon,\quad \text{and}\quad \|x-y\|<\varepsilon.$$
\end{definition}

As usual, \cite[Lemma 3.7]{Chak20} yields an equivalent version of this property. We will include the details of the proof for this one.

\begin{proposition}\label{remark:equiv-def}
The pair of Banach spaces $(X,Y)$ has the bounded BPBp-$m$ if and only if given $\varepsilon>0$ and $R\geq 1$ there exists some $\eta(\varepsilon, R)>0$ such that whenever $x\in S_X$ and $T\in\mathcal{L}(X,Y)$ with $\|T\|\leq R$ and $m(T)= 1$ satisfy $\|T(x)\|<1+\eta(\varepsilon, R)$, there exist $S\in\mathcal{L}(X,Y)$ and $y\in S_X$ such that
$$\|S(y)\|=m(S)=m(T)=1,\quad \|T-S\|<\varepsilon,\quad \text{and}\quad  \|x-y\|<\varepsilon.$$
\end{proposition}

\begin{proof}
We prove both implications. Let $(X,Y)$ have the bounded BPBp-$m$ with the function $\eta(\varepsilon, R, \rho)$. Define the mapping $\eta_2(\varepsilon, R):=\eta(\varepsilon, R, 1)$. Fix $\varepsilon>0$ and $R>0$. Let $T\in\call(X, Y)$ and $x\in S_X$ be such that $m(T)=1$, $\|T\|\leq R$, and $\|T(x)\|<1+\eta_2(\varepsilon, R)=1+\eta(\varepsilon, R, 1)$. Then there exist $S\in\mathcal{L}(X,Y)$ and $y\in S_X$ such that
$$\|S(y)\|=m(S)=m(T)=1,\quad \|T-S\|<\varepsilon,\quad \text{and}\quad  \|x-y\|<\varepsilon.$$

Conversely, let $(X,Y)$ satisfy the alternative definition with the function $\eta(\varepsilon, R)$. Define the mapping $\eta_2(\varepsilon, R, \rho):=\eta\left( \min\left\{ \varepsilon, \frac{\varepsilon}{\rho} \right\}, \frac{R}{\rho} \right)$, if $\rho>0$, and $\eta_2(\varepsilon, R, 0):=\varepsilon$. We distinguish 2 cases now.

\textit{Case 1}: $\rho=0$. Fix $\varepsilon>0$ and $R\geq 0$. Let $T\in\call(X, Y)$ and $x\in S_X$ be such that $m(T)=0$, $\|T\|\leq R$, and $\|T(x)\|<\eta_2(\varepsilon, R, 0)= \varepsilon$. Then, by \cite[Lemma 3.7]{Chak20}, there exist $S\in\call(X,Y)$ and $y=x\in S_X$ such that $m(S)=m(T)=0=\|S(y)\|$, $\|S-T\|<\varepsilon$, and $\|x-y\|=0<\varepsilon$.

\textit{Case 2}: $\rho>0$. Fix $\varepsilon>0$ and $R\geq 0$. Let $T\in\call(X, Y)$ and $x\in S_X$ be such that $m(T)=\rho$, $\|T\|\leq R$, and $\|T(x)\|<\rho\cdot (1+\eta_2(\varepsilon, R, \rho))$. Then $m(\frac{T}{\rho})=1$, and $\|\frac{T}{\rho}(x)\|<1+\eta_2(\varepsilon, R, \rho)\leq 1+\eta\left( \min\left\{ \varepsilon, \frac{\varepsilon}{\rho} \right\}, \frac{R}{\rho} \right)$. To simplify, call $\delta:=\min\left\{ \varepsilon, \frac{\varepsilon}{\rho} \right\}$. Thus, there are $S\in \call(X, Y)$ and $y\in S_X$ such that $m(S)=m(\frac{T}{\rho})=1=\|S(y)\|$, $\|S-\frac{T}{\rho}\|<\delta$, and $\|x-y\|<\delta$. Then the operator $\rho S$ satisfies that $m(\rho S)=m(T)=\rho=\|\rho S(y)\|$, $\|x-y\|<\delta\leq \varepsilon$, and $\|\rho S-T\|<\rho\delta\leq \varepsilon$.

This shows that $(X, Y)$ has the bounded BPBp-$m$.\qedhere
\end{proof}

Note that the following implications hold trivially from the definitions, and they provide several pairs of Banach spaces satisfying the \textbf{L}$_o$-$m$.

\begin{proposition}
Let $X$ and $Y$ be Banach spaces.
\begin{enumerate}
\item[a)] If $(X,Y)$ has the BPBpp-$m$, then it has the bounded BPBp-$m$.
\item[b)] If $(X, Y)$ has either the bounded BPBp-$m$ or the \textbf{L}$_{o,o}$-$m$, then it has the \textbf{L}$_o$-$m$.
\item[c)] If $(X,Y)$ has the \textbf{L}$_o$-$m$, then $\ma(X,Y)$ is dense in $\call(X,Y)$.
\end{enumerate}
\end{proposition}

Also, note that Theorem \ref{theorem:ma-not-loom} implies, in particular, that there are pairs of Banach spaces $(X,Y)$ satisfying the bounded BPBp-$m$ (in particular, the \textbf{L}$_o$-$m$) but failing the \textbf{L}$_{o,o}$-$m$, despite the fact that every operator attains its minimum.

Theorems 3.5 and 3.6 of \cite{Chak21} imply, in particular, that pairs of finite-dimensional Banach spaces $(X,Y)$ always satisfy the \textbf{L}$_o$-$m$, but by the nature of its proof by contradiction, the $\eta$ depends on each particular operator $T$ a priori. We show now that, actually, these pairs of spaces have the bounded BPBp-$m$. The proof is based on that of \cite[Theorem 2.1]{AAGM08}.

\begin{proposition}\label{prop-fin-dim-bd-bpbp-m}
If $Y$ is a finite-dimensional Banach space and $X$ is any Banach space, then $(X, Y)$ has the bounded BPBp-$m$.
\end{proposition}

\begin{proof}
Note that the case where $X$ is infinite-dimensional is a clear consequence of \cite[Lemma 3.7]{Chak20}, since no operator is injective, and so the pair $(X,Y)$ has the BPBpp-$m$. So assume now that $X$ is finite-dimensional.

We will argue by contradiction. Note that the case $\rho=0$ is already solved at \cite[Lemma 3.7]{Chak20} with the uniform choice $\eta=\varepsilon$. So we only need to prove this for positive $\rho$. Suppose that this is not the case. Then, there exist $\varepsilon>0$, $R>0$, and $0<\rho\leq R$ such that for every $n\in \bbn$, there are an operator $T_n\in \call(X, Y)$ with $\|T_n\|\leq R$ and $m(T_n)=\rho$ and a point $x_n\in S_X$ such that $\|T_n(x_n)\|\leq \rho\cdot (1+\frac{1}{n})$, but such that for every operator $S\in \call(X, Y)$ with $m(S)=\rho$ and $\|S-T\|< \varepsilon$ and for every $y\in S_X$ with $\|x_n-y\|<\varepsilon$, we have that $\|S(y)\|>\rho$. Now, by compactness, note that up to a subsequence, $(T_n)_n$ and $(x_n)_n$ respectively converge to some $S\in \call(X, Y)$ and $y\in S_X$ such that $\|S(y)\|=m(S)=\rho$. But note that if $n$ is big enough, then $\|S-T_n\|<\varepsilon$ and $\|y-x_n\|<\varepsilon$, which is a contradiction with the assumption.
\end{proof}

Note (see for instance \cite[Lemma 3.2]{Chak20}) that an invertible operator $T\in\call(X,Y)$ attains its norm at $x\in S_X$ if and only if its inverse attains its minimum at $\frac{T(x)}{\|T(x)\|}$. This motivates wondering if there is a relation between a pair $(X,Y)$ satisfying the BPBp and the pair $(Y,X)$ satisfying the bounded BPBp-$m$. Proposition \ref{prop-fin-dim-bd-bpbp-m} shows that if $Y$ is a strictly convex space that is not uniformly convex, then $(Y, \ell_1^2)$ has the bounded BPBp-$m$, but the pair $(\ell_1^2, Y)$ cannot have the BPBp (see \cite[Corollary 3.3]{ACKLM15}). However, about the other implication, we can show the following.

\begin{theorem}\label{theo:bpbp-then-bpbpm}
Let $X$ and $Y$ be Banach spaces such that given $\varepsilon\in (0,1)$, there is some $\eta=\eta(\varepsilon)\in (0,1)$ such that for every monomorphism $T\in\bl(X, Y)$, the pair $(T(X), X)$ has the BPBp with $\eta$. Then $(X, Y)$ has the bounded BPBp-$m$.
\end{theorem}

\begin{proof}
Let $\varepsilon\in (0,1)$, and let $0\leq \rho\leq R$ be given. Note that if $\rho=0$, then using $\delta(\varepsilon)=\varepsilon$, the claim is already contained in the proof of \cite[Lemma 3.7]{Chak20}. So assume now that $\rho>0$. Let $0<\varepsilon_1=\varepsilon_1(\varepsilon)$ be small enough so that $0<4\varepsilon_1 < \frac{\varepsilon}{2}$, and
$$0<\frac{R^2 \varepsilon_1 /\rho}{1-\varepsilon_1 R/\rho} < \varepsilon.$$
Let 
$$\delta=\delta(\varepsilon, R, \rho):=\min\left\{\frac{\varepsilon}{2},\, \frac{\eta(\varepsilon_1) \rho}{1-\eta(\varepsilon_1)} \right\}.$$
Let $T\in \call(X, Y)$ be such that $\rho=m(T)\leq \|T\|\leq R$, and let $x\in X$ be such that $\|T(x)\| < m(T) (1+\delta)\leq m(T) (1+ \frac{\eta(\varepsilon_1) \rho}{1-\eta(\varepsilon_1)})$. Note that $T:X\rightarrow T(X)$ is an isomorphism. Denote $Z:=T(X)$. Then, if $U:=\frac{T^{-1}}{\|T^{-1}\|}$ and $x_0:=\frac{T(x)}{\|T(x)\|}$, we have
$$\left\|U \left( x_0 \right) \right\|\geq (1-\eta(\varepsilon_1)).$$
Therefore, there exist $z_0\in S_Z$ and $S\in S_{\call(Z, Y)}$ such that
$$\|S(z_0)\|=\|S\|,\quad \|S - U\| < \varepsilon_1 ,\quad \text{and}\quad \|x_0 - z_0\|<\varepsilon_1.$$
Now, since $\alpha:=\|U-S\|<\varepsilon_1$, then $\|U^{-1} S - \id_X\|<\|U^{-1}\|\alpha$, and using Newman series as in \cite[Page 193]{TL86}, we get
$$\|S^{-1} - U^{-1}\| \leq \|U^{-1}\|^2 \frac{\alpha}{1-\|U^{-1}\| \alpha} < \frac{(R/\rho)^2 \varepsilon_1}{1-\varepsilon_1 R/\rho},$$
since $\|U^{-1}\|= \|T^{-1}\| \|T\|\leq R/\rho$. Moreover, since $S$ attains its norm at $z_0$, the scalar multiples of $S^{-1}$ attain their minimum at $\frac{S(z_0)}{\|S(z_0)\|}=S(z_0)$. Note also that $\|S^{-1} - U^{-1}\|<\frac{(R/\rho)^2 \varepsilon_1}{1-\varepsilon_1 R/\rho}$, so in particular, since $T=mU^{-1}$, we have
$$\|\rho S^{-1} - T\|<\frac{R^2 \varepsilon_1 /\rho}{1-\varepsilon_1 R/\rho}<\varepsilon,$$
and clearly $m(\rho\cdot S^{-1})=\rho$. Finally, it remains to show that $\frac{S(z_0)}{\|S(z_0)\|}$ is close to $x$. Notice that $T(x)=x_0 \|T(x)\|$, and $z_0=T(t)$ for some $t\in X$. First note that
\begin{equation}\label{eq:b1}
\Big\|x - t\|T(x)\|\Big\| \leq \|T(x)\| \| T^{-1} \| \varepsilon_1 < (1+\delta) \varepsilon_1 < 2\varepsilon_1,
\end{equation}
since $1\leq \|T(x)\| \| T^{-1} \| < 1+\delta$. Now, since $t=U(z_0)\|T^{-1}\|$, we also have
\begin{equation}\label{eq:b2}
\Big\|\|T(x)\| \|T^{-1}\| U(z_0) - \|T(x)\| \|T^{-1}\| S(z_0)\Big\|\leq \|T(x)\| \| T^{-1} \| \varepsilon_1 < (1+\delta) \varepsilon_1 < 2\varepsilon_1.
\end{equation}
Finally, note that $\|S(z_0)\|=\|S\|=1$, so
\begin{equation}\label{eq:b3}
\Big\|\|T(x)\| \|T^{-1}\| S(z_0) - S(z_0)\Big\| = \left| \|T(x)\| \|T^{-1}\| - 1 \right| < \delta \leq \frac{\varepsilon}{2}.
\end{equation}

Using \eqref{eq:b1}, \eqref{eq:b2}, and \eqref{eq:b3}, by the triangle inequality we get
\begin{equation*}
\left\| x - \frac{S(z_0)}{\|S(z_0)\|}\right\| < 2\varepsilon_1 + 2\varepsilon_1 + \frac{\varepsilon}{2} < \varepsilon.\qedhere
\end{equation*}
\end{proof}


There are some immediate consequences of the previous result. First, recall that a Banach space $X$ has property $\beta$ of Lindenstrauss if there are two sets $\{x_\alpha:\, \alpha\in \Lambda\}\subset S_X$ and $\{x^*_\alpha:\, \alpha\in\Lambda\}\subset S_{X^*}$ and a number $0\leq \rho < 1$ such that the following conditions hold:
\begin{enumerate}
\item $x^*_\alpha(x_\alpha)=1$ for all $\alpha\in\Lambda$,
\item $|x^*_\alpha(x_\gamma)|\leq \rho<1$ if $\alpha,\gamma\in\Lambda$ are such that $\alpha\neq \gamma$,
\item $\|x\|=\sup_{\alpha\in\Lambda} |x^*_\alpha(x)|$, for all $x\in X$.
\end{enumerate}
This property is satisfied, for instance, for finite-dimensional polyhedral spaces, and for Banach spaces $X$ such that $c_0\subset X\subset \ell_\infty$, and property $\beta$ of Lindenstrauss is known to imply property $B$ of Lindenstrauss (\cite[Proposition 3]{Lindenstrauss63}). Moreover, if $X$ has property $\beta$ with constant $\rho<1$, then for all Banach spaces $Y$, the pair $(Y,X)$ has the BPBp with an $\eta$ uniquely determined by $\rho$ (\cite[Theorem 2.2]{AAGM08}), so we actually get the following.

\begin{corollary}
If $X$ and $Y$ are Banach spaces and $X$ has property $\beta$ of Lindenstrauss, then $(X, Y)$ has the bounded BPBp-$m$.
\end{corollary}

In particular, \cite[Theorem 1]{Partington82} implies that every Banach space can be equivalently renormed to be a universal domain for the bounded BPBp-$m$.

\begin{corollary}
Given any Banach space $X$, there exists a Banach space $Z$ isomorphic to $X$ such that $(Z,Y)$ has the bounded BPBp-$m$ for all Banach spaces $Y$.
\end{corollary}

Similarly, note that subspaces of uniformly convex Banach spaces are also uniformly convex, and the modulus of convexity of the subspace is bigger or equal than that of the original space. Therefore, due to the computations from \cite[Theorem 3.1]{KL14}, we get the following result, which significantly improves \cite[Theorems 3.5 and 3.8]{BR21a}.

\begin{corollary}\label{cor:Y-un-conv}
If $X$ and $Y$ are Banach spaces and $Y$ is uniformly convex, then $(X, Y)$ has the bounded BPBp-$m$.
\end{corollary}

The Banach spaces $Y$ such that $(\ell_1^2,Y)$ has the BPBp are characterized as those that satisfy a geometrical property called the \textit{Approximate Hyperplane Series Property for pairs} (see \cite{AAGM08,ACKLM15}). On the other hand, it has been shown already that if $X$ is any Banach space, then $(X, \ell_1^2)$ satisfies the \textbf{L}$_{o,o}$-$m$ and the bounded BPBp-$m$. Moreover, note that if $\dim(X)\neq 2$, then $(X, \ell_1^2)$ has the BPBpp-$m$ by \cite[Lemma 3.7]{Chak20}, so the following question is natural: is it true for every Banach space $X$ that $(X, \ell_1^2)$ has the BPBpp-$m$? Also, note that if $X$ and $Y$ are both finite-dimensional, then $(X, Y)$ has both the \textbf{L}$_{o,o}$-$m$ and the bounded BPBp-$m$. This motivates to wonder the following: is it true for every pair of finite-dimensional Banach spaces $X$ and $Y$ that $(X, Y)$ has the BPBpp-$m$? We will answer these questions in the negative.

\begin{theorem}\label{theorem:ell22,ell12}
The pair $(\ell_2^2, \ell_1^2)$ fails to have the BPBpp-$m$ in the real case.
\end{theorem}

\begin{proof}
Note that an isomorphism $S\in\call(\ell_1^2, \ell_2^2)$ can only attain its norm at the points $(1,0)$, $(-1,0)$, $(0,1)$, and $(0,-1)$, as $\ell_2^2$ is strictly convex. For each $n\in\bbn$, we will define a linear operator $T_n\in\call(\ell_2^2, \ell_1^2)$ that maps $S_{\ell_2^2}$ into the unique ellipse containing the points $(1,0)$, $(-1,0)$, $(0,1)$, $(0,-1)$, and $\left( \frac{1}{2}+\frac{1}{n}, \frac{1}{2}+\frac{1}{n} \right)$ (see Figure~\ref{fig:elipse}), which satisfies the equation
$$-((-8-8n+2n^2)xy)+(2+n)^2(1-x^2-y^2)=0.$$
Moreover, we will make sure that 
$$T_n(e_1)=\left( \frac{1}{2}+\frac{1}{n}, \frac{1}{2}+\frac{1}{n} \right)$$ 
and 
$$T_n(e_2)=\left( \frac{-(n+2)}{4\sqrt{n+1}}, \frac{(n+2)}{4\sqrt{n+1}} \right),$$
which belong to that ellipse indeed. In other words, for each $(x,y)\in\ell_2^2$, we define
$$T_n(x,y):=\left( \frac{n+2}{2n}x - \frac{(n+2)}{4\sqrt{n+1}}y, \frac{n+2}{2n}x + \frac{(n+2)}{4\sqrt{n+1}}y \right).$$

If we solve the corresponding systems of linear equations, we can easily verify that the preimages by $T$ of the points 
$(1,0)$, $(-1,0)$, $(0,1)$, and $(0,-1)$ of $B_{\ell_1^2}$ are, respectively, the points $\left( \frac{n}{n+2},-\frac{2\sqrt{n+1}}{n+2} \right)$, $\left( -\frac{n}{n+2},\frac{2\sqrt{n+1}}{n+2} \right)$, $\left( \frac{n}{n+2},\frac{2\sqrt{n+1}}{n+2} \right)$, and $\left( -\frac{n}{n+2},-\frac{2\sqrt{n+1}}{n+2} \right)$, which clearly belong to $B_{\ell_2^2}$. Note that, by construction, $T_n(B_{\ell_2^2})$ is a region (delimited by an ellipse) that contains $B_{\ell_1^2}$, and the ellipse $T_n(S_{\ell_2^2})$ intersects $S_{\ell_1^2}$ only at the points $(1,0)$, $(-1,0)$, $(0,1)$, and $(0,-1)$ (due to its strict convexity). Thus, it is geometrically clear (and can be checked analytically) that $m(T_n)=1$, and that $\|T_n(e_1)\|=2\left(\frac{n+2}{2n}\right)=1+\frac{2}{n}$.

\begin{figure}[H]
\centering
\includegraphics[width=\textwidth]{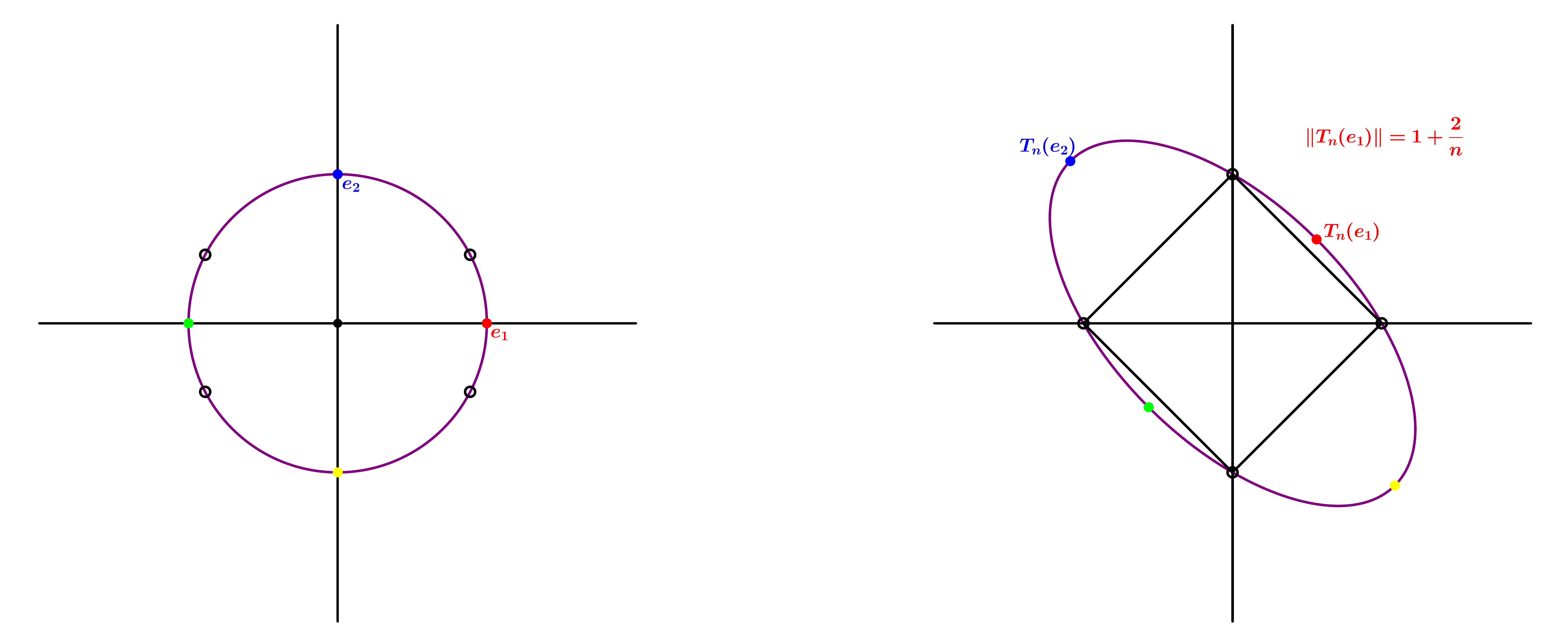}
\caption{Construction of $T_n\in\call(\ell_2^2, \ell_1^2)$}
\label{fig:elipse}
\end{figure}

Now, we will argue by contradiction. Suppose that $(\ell_2^2, \ell_1^2)$ has the BPBpp-$m$. Then for all $\varepsilon\in (0,1)$, there must be some $\eta(\varepsilon)\in (0,1)$ (converging to $0$ with $\varepsilon$) such that whenever an operator $T\in\call(\ell_2^2, \ell_1^2)$ and $x\in S_{\ell_2^2}$ satisfy that $m(T)=1$ and $\|T(x)\|<1+\eta(\varepsilon)$, then there is an operator $S\in\call(\ell_2^2, \ell_1^2)$ satisfying that
$$m(S)=\|S(x)\|=m(T)=1,\quad \text{and}\quad \|S-T\|<\varepsilon.$$
Fix $0<\varepsilon<1$ and suppose that such $\eta(\varepsilon)$ exists. There must exist some $n\in\bbn$ such that $\eta(\varepsilon)\geq \frac{2}{n}$. Take now $T=T_n$ and $x=(1,0)\in B_{\ell_2^2}$, and let $S$ be given as before. Since $S$ is an isomorphism that attains its minimum at $(1,0)$, in particular $S^{-1}$ is an isomorphism that attains its norm at the point $\frac{S((1,0))}{\|S((1,0)))\|}$. But $S^{-1}$ can only attain its norm at the points $(1,0)$, $(-1,0)$, $(0,1)$, and $(0,-1)$ of $B_{\ell_1^2}$. In particular, $S((1,0))$ must be one of those points. Note now that the distances from those $4$ points to $T_n((1,0))=\left( \frac{1}{2}+\frac{1}{n}, \frac{1}{2}+\frac{1}{n} \right)$ are bounded below by $1$, so we must have
$$\varepsilon>\|T-S\|\geq \|(T-S)(1,0)\|\geq 1,$$
which is a contradiction. Since this can be done for every $n\in\bbn$ and every $\varepsilon\in (0,1)$, the claim is proven.
\end{proof}

In particular, we get the following.

\begin{corollary}
There are Banach spaces $(X,Y)$ satisfying the \textbf{L}$_{o,o}$-$m$ and the bounded BPBp-$m$, but not the BPBpp-$m$.
\end{corollary}

As a final remark, in the case of the norm, recall that there is a natural strengthening of the \textbf{L}$_{o,o}$ considered in the literature, namely the Bishop-Phelps-Bollob\'as operator property (or just BPBop), which is defined like the \textbf{L}$_{o,o}$ but where $\eta=\eta(\varepsilon)$ does not depend on the operator. This property is known to be very restrictive:
\begin{enumerate}
\item $(X,\bbk)$ has the BPBop if and only if $X$ is uniformly convex (see \cite[Theorem 2.1]{KL14}).
\item $(X,Y)$ always fails the BPBop if $X$ and $Y$ both have dimension at least $2$ (see \cite[Theorem 2.1]{DKKLM20}).
\end{enumerate}

It is natural to wonder if we can also consider a uniform strengthening of the \textbf{L}$_{o,o}$-$m$ similarly, but it turns out that this property is even more restrictive than its norm counterpart.

\begin{proposition}
Let $X$ be a Banach space of dimension at least $2$. Then there is no nontrivial Banach space $Y$ such that the following property holds: ``for every $\varepsilon\in (0,1)$ and every $\rho\geq 0$, there exists $\eta(\varepsilon, \rho)\in (0,1)$ (converging to $0$ with $\varepsilon$) satisfying that whenever $T\in\call(X,Y)$ with $m(T)=\rho$ and $x\in S_X$ are such that $\|T(x)\|<\rho (1+\eta(\varepsilon,\rho))$, if $\rho>0$, or $\|T(x)\|<\eta(\varepsilon, \rho)$, if $\rho=0$, there exists $y\in S_X$ such that $m(T)=\|T(y)\|=\rho$, and $\|x-y\|<\varepsilon$''.
\end{proposition}

\begin{proof}
Given $\varepsilon\in (0,1)$, suppose that $\eta(\varepsilon,0)\in (0,1)$ exists. 

Fix $\varepsilon<1$. Let $T\in\call(X,Y)$ be a non-injective (and non-zero) operator such that $\|T\|\leq \frac{\eta(\varepsilon, 0)}{2}$. By Riesz' lemma, there is some $u\in S_X$ such that $\operatorname{dist}(u, \operatorname{Ker}(T))>\varepsilon$. However, as $\|T(u)\|\leq \frac{\eta(\varepsilon, 0)}{2}<\eta(\varepsilon, 0)$, there must be a point $y$ such that $\|u-y\|<\varepsilon$ and $\|T(y)\|=0$. This is a contradiction.
\end{proof}

\noindent\textbf{Acknowledgements:} \\
All authors have been supported by MICIU / AEI / 10.13039 / 501100011033 and ERDF/EU through the grants PID2021-122126NB-C33 (first, second, and fourth author) and PID2021-122126NB-C31 (third author). The first and second authors have also been supported by PROMETEU/2021/070. The third author has also been supported by “Maria de Maeztu” Excellence Unit IMAG, reference CEX2020-001105-M funded by MICIU/AEI/10.13039/501100011033. The fourth author has also been supported by the Basic Science Research Program, National Research Foundation of Korea (NRF), Ministry of Education, Science and Technology [NRF-2020R1A2C1A01010377].

\end{document}